\documentclass[10pt,a4paper]{article}
\usepackage{amsmath,amsxtra,amssymb,latexsym,amscd,amsfonts,multicol,enumerate,ifthen,indentfirst,amsthm,amstext}
\usepackage[mathscr]{eucal}
\usepackage{graphics, graphpap}
\usepackage[pdftex]{graphicx}
\usepackage{xspace}
\usepackage{array, tabularx, longtable}
\usepackage{multicol,color}
\usepackage{mathrsfs}
\usepackage{makeidx}
\usepackage{indentfirst}
\usepackage{subfigure}
\usepackage{tikz}
\usepackage{hyperref}
\usepackage[T1]{fontenc}
\usepackage{authblk}

 \def \br {\overset{\circ}{B} } 
 
 \def \sr {S_c } 

\belowdisplayshortskip =\abovedisplayshortskip
\def\NoBlackBoxes{\overfullrule=0pt }
\NoBlackBoxes
\theoremstyle{plain}
\newtheorem{theorem}{Theorem}
\newtheorem{proposition}{Proposition}
\newtheorem{lemma}{Lemma}
\newtheorem{definition}{Definition}

\newcommand\R{{\mathbb R}}
\newcommand\var{{\textnormal Var}}

\newcommand\E{\mathbb{E} }
\newcommand \PP{\mathbb{P}}

\title{ \bf Asymptotic formula for the tail of the maximum of smooth  Stationary Gaussian fields on non locally convex sets}
\author[1]{ Jean-Marc Aza\"\i s} 
\author[2]{Viet-Hung Pham}

\affil[1]{\small{Institut de Math\'{e}matiques de Toulouse\\ Universit\'{e} Paul Sabatier (Toulouse III), France}}
\affil[2]{\small{Department of Mathematics and Informatics\\ Hanoi National University of Education, Vietnam}}
\date{}

\begin{document} 
\maketitle

\abstract {In this paper we consider the distribution of the maximum of  a Gaussian field defined on non locally convex sets. Adler and Taylor or Aza\"\i s and Wschebor give the expansions in the locally convex case. The present paper generalizes their results to the non locally  convex case  by giving a full expansion in dimension 2 and some generalizations in  higher dimension. For  a given class of sets, a Steiner formula is established  and the correspondence between  this formula and the tail of the maximum is proved. The main tool is a recent result of Aza\"\i s and Wschebor that shows that under some conditions the excursion set is close to a ball with  a random radius. Examples are given in dimension 2 and higher.} \medskip

\textbf{Key-words:} Stochastic processes, Gaussian fields, Rice formula, distribution of the maximum, non locally convex indexed set.\\
\indent \textbf{Classifications:} 60G15, 60G60, 60G70.

\section{Introduction}
 Let $ \mathcal{X} = \{X(t) : t \in S\subset \R^n  \}$  be a random field with real values and let $M_S $ be its maximum (or supremum) on $S$.  Computing the distribution  of the maximum 
 is a very important issue  from the theoretical point of view and  also  has a great impact on applications, especially in spatial statistics. This problem has therefore received a great deal  of attention from many authors.

However
an exact result  is known only in very few cases, (see Aza\"{\i}s  and Wschebor \cite {5}).  In other cases,    the only available results are  asymptotic expansions or bounds  mainly   in the case of stationary Gaussian random fields. 

One of the most well-known and quite general methods is the "double-sum method", first  proposed by Pickands \cite{pic} and extended by Piterbarg \cite{10}, \cite{11}. The main idea of this method is to use the inclusion-exclusion principle and the Bonferroni inequality   after  dividing the parameter set into  suitable smaller subsets.  It was first proposed in dimension 1 : $n=1$  (in this case we use the classical  terminology of  "random processes" instead of "random field").  More precisely,  for   some particular processes, i.e.,  the "$\alpha$  processes", Pickands proposed  an equivalent  for the tail of the maximum.  However, the result depends  on some unknown constants, referred to as   Pickands' constants and just gives  an equivalent. 
 
Another method is the "tube method" proposed  by Sun \cite{13}. She observed that if the Karhunen-Lo\`{e}ve expansion of the field is finite in the sense that there exist a finite number of random variables $\xi_1,\ldots,\xi_k \overset{i.i.d}{\sim} \mathcal{N}(0,1)$ such that at every point $t$ in the parameter set, the value of the field at this point $X(t)$ can be expressed as
$$X(t)=a_t^1\xi_1+\ldots+a_t^k\xi_k,$$
where the vector $(a_t^1,\ldots,a_t^k)$ has unit  norm since $\textrm{Var}(X(t))=1$, then the original parameter set can be transformed into a subset of the unit sphere $\mathcal{S}^{k-1}$. She then  used  Weyl's tube formula to compute  the polynomial expansion of the volume of the tube around a subset of the unit sphere and derived  the asymptotic formula of the tail of the distribution of the maximum from this expansion. She is the first one who realizes the strong connection between the geometric functionals of the parameter set (the coefficients of the polynomial expansion) and the tail of the distribution. When the Karhunen-Lo\`{e}ve expansion is not finite, she uses a truncation argument to derive an asymptotic formula with two terms. Later on, this method was extended by Takemura and Kuriki  \cite{14}, \cite{15}.

In the 1940s, in his pioneering work, Rice \cite{rice} considered a stationary process $\mathcal{X}$  with  $\mathcal{C}^1$ paths defined on the  compact  interval $[0,T]$. He observed that for every level $u$:
\begin{displaymath}
\begin{array}{rl}
\displaystyle \PP\left(\underset{t\in [0,T]}{\max}X(t)\geq u\right) & \displaystyle \leq \PP(X(0)\geq u)+\PP\left(\exists t \in [0,T]:\, X(t)=u,X'(t)\geq 0\right)\\
& \displaystyle \leq \PP(X(0)\geq u)+\E\left(\textrm{card}\left\{ t \in [0,T]:\, X(t)=u,X'(t)\geq 0\right\}\right),\\
\end{array}
\end{displaymath}
 where the last expectation can be  evaluated  by the famous Rice-Kac formula. 
 This upper bound was later proved to be sharp by Piterbarg \cite{12}. This Rice-Kac formula is the starting point of the following methods dealing with the random fields: the  "Rice method" by Aza\"\i s and Delmas \cite{3}, \cite{8}, the "direct method" by Aza\"\i s and Wschebor \cite{5} and the  "Euler characteristic method" by Adler and Taylor \cite{1}. These methods use a multidimensional Rice-Kac formula : Generalized Rice formula (Aza\"\i s and Wschebor) or Metatheorem (Adler and Taylor).

In the direct method, Aza\"\i s and Wschebor used some results  from the random matrix theory to compute the  expectation of the absolute value  of the determinant of the Hessian that  appears  in the Rice formula.   They obtained an upper bound for the tail of the distribution depending  on some  geometric functionals of the parameter set. This upper bound is also  sharp.

Adler and Taylor combined differential and integral geometry to find the "Euler characteristic method" that gives one of most  frequently used  results in this area. They considered stratified sets, i.e. locally  convex Whitney stratified  manifolds.
First, they used the Metatheorem to compute the expectation of the Euler characteristic of the excursion set (see Theorem 12.4.1) and, second, they proved that the difference between the above expectation and the excursion probability (the tail of the distribution) is super exponentially smaller (see Theorem 14.3.3). Note that the geometric functionals of the parameter set appear in the expectation of the Euler characteristic under the name  of Lipschitz-Killing curvatures. 

We recall an important example when the parameter set $S$ is a convex body in $\R^2$ (compact, convex, with non-empty interior) and $\mathcal{X}$ is an isotropic centered Gaussian field defined on some  neighborhood of $S$   and satisfying  $\var(X(t) =1 $ and  $\textrm{Var}(X'(t))=I_n$, where $I_n$ is the identity matrix of size $n$.  Let us denote:
$$M_S =\max_{t\in S} X(t).$$
Then, under some regularity conditions, the Euler characteristic method gives:  \begin{equation}\label{exp1}
\PP(M_S\geq u)= \overline{\Phi}(u) + \frac{\sigma_1(\partial S)}{2\sqrt{2\pi}}\varphi(u)+\frac{\sigma_2(S)}{2\pi}u\varphi(u)+o\left(\varphi\left((1+\alpha)u\right)\right),
\end{equation} for some $\alpha >0$,
where $\overline{\Phi}(u)$ and $\varphi(u)$ are the tail distribution and the density of a standard normal variable, $\sigma_2(S)$ is the area of $S$ and $\sigma_1(S)$ is the perimeter of $S$. Note  that  the coefficient $1$ of the term $\overline{\Phi}(u)$ can be interpreted as the Euler characteristic of $S$.

 Adler and Taylor use the local convexity that can be defined as the fact that for every point $t\in S$, the support cone $C_t$
generated by the set of directions
 \begin{equation}\label{loco}
 \Biggl\{ \lambda \in \R^2 :\; \|\lambda\|=1,\, \exists s_n \in S \; \textnormal{such\; that}\; s_n\rightarrow t \; \textnormal{and} \; \frac{s_n-t}{\|s_n-t\|}\rightarrow \lambda \Biggr\},
 \end{equation}
is convex, plus some regularity conditions (see, for example \cite[Section 8.2]{1}) ($\|.\|$ is the Euclidean norm). Similarly, Aza\"\i s and Wschebor \cite[p. 231]{6} use the condition:
\begin{equation}\label{kap}
\kappa (S)= \underset{t\in S}{\sup} \underset{s\in S,\; s\neq t}{\sup} \frac{\textnormal{dist}(s-t,C_t)}{\|s-t\|^2}<\infty
\end{equation}
where dist  is the Euclidean distance. $1/\kappa(s) $ is called the reach (Federer \cite{federer}; Takemura and Kuriki  \cite{14}).

However none of these methods is able to provide a full expansion for the asymptotic formula in the non-locally convex cases, even the very simple case of $S$ being "the angle" that is the union of two segments with the angle $\beta\in (0,\pi)$, see Figure \ref{figu1},
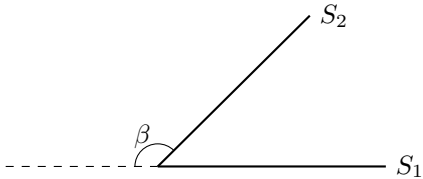
\begin{figure}[h]
\centering
\begin{tikzpicture}
\draw[thick] (0,0) -- (3,0) node[right]{$S_1$};
\draw[thick] (0,0) -- (2,2) node[right]{$S_2$};
\draw (-0.3,0) arc (180:45:0.3);
\draw (-0.2,0.4) node{$\beta$};
\draw[dashed] (-2,0)--(0,0);
\end{tikzpicture}
\caption{The angle, an example of non-local convexity.} \label{figu1}
\end{figure}
which is presented in \cite[Section 14.4.4] {1}. By a full expansion, we mean  a formula   of the type  (\ref{exp1}) with  three terms in dimension 2 and $n+1$ in the general case. 
 
We are therefore  interested in the following question:

\textit{"Can we find  some   full expansions  for the  tail of the maximum in  some  non-locally convex cases in dimension $2$ and higher?"} 

In a previous article  \cite{4}, we gave an upper bound for the tail of the distribution for quite general parameter sets $S$. More precisely,  if  $S$ is the Hausdorff limit of connected polygons $S_n$, if $\mathcal{X}$ is a stationary centered Gaussian field with variance $1$ and $\textrm{Var}(X'(t))=I_n$ defined on a neighborhood of $S$ then for every level $u$:

\begin{equation}\label{bou1}
\PP \{M_S \geq u\} \leq \overline{\Phi}(u)+ \frac{\liminf_n \sigma_1( S_n) \varphi (u)}{2\sqrt{2\pi}}+\frac{\sigma_2(S)}{2\pi }\left[ c\varphi (u\textrm{/}c)+u\Phi (u\textrm{/c})\right]\varphi (u),
\end{equation}
where $c=\sqrt{\textnormal{Var} (X''_{11}(t))-1}$,    $X''_{11}(t) =\frac{\partial^2 X(t)}{ \partial t^2_1}$ ,
   $\sigma_2$ is the area and $\sigma_1$ is the perimeter.

Note that (\ref{bou1}) can be applied to  polygons  taking the simpler form:
\begin{equation}\label{f:bou2}
\PP \{M_S \geq u\} \leq \overline{\Phi}(u)+ \frac{\sigma_1( S) \varphi (u)}{2\sqrt{2\pi}}+\frac{\sigma_2(S)}{2\pi }\left[ c\varphi (u\textrm{/}c)+u\Phi (u\textrm{/c})\right]\varphi (u),
\end{equation}
When the polygon is convex,  we can check that (\ref{f:bou2}) is sharp by comparing  (\ref{exp1}) and (\ref{f:bou2}).  However  we do not have such  information  in the non-convex case. 
 
Recently Aza\"\i s and Wschebor \cite{7} proposed a new method, still based on the generalized Rice formula, to derive the asymptotic formula when the parameter set $S$ is fractal. They also gave an asymptotic expansion with two terms in the case of a parameter set with a finite perimeter (defined as  an outer Minkowsky content). However, for example in dimension 2, this result does not give  the coefficient of 
 $\overline{\Phi}(u)$, which is the third term by  order of importance. 

 Section \ref{sect2}  is devoted to dimension 2.   We define a quite general class of parameter sets in $\R^2$ (see Definition \ref{d:2}) and derive the asymptotic formula for the tail of the maximum of the random fields defined on these parameter sets. This is our main result (Theorem \ref{theo}).  It shows that the coefficient corresponding to $\overline{\Phi}(u)$ is \textbf{not} always equal   to the Euler characteristic of the parameter set and, in fact, it is derived from the Steiner formula that gives the volume (area) of the tube around $S$. Here again, we emphasize the strong connection between the tube formula of the parameter set and the tail of the maximum.

In Section \ref{sect3}, we examine this connection by considering some examples.  We  use elementary geometry to compute the tube formula,  obtain the geometric functionals, and then  immediately   obtain the asymptotic expansion of  the tail distribution. All the examples correspond to new results. In particular, the examples in Subsection \ref{sub35} and \ref{sub36} could shed  new light on this problem. We also  conjecture that the strong connection still occurs in  dimensions  higher than $2$ and 3,  and even in fractal dimension. 

\subsection*{Hypotheses and notation}
 We will use the following assumption on the random field $\mathcal{X}$ throughout this  paper:
 
 \textbf{Assumption $A$: }$\mathcal{X}$ is a random field defined on a ball $B\subset \R^n$ satisfying:
\begin{itemize}
\item[i.] $\mathcal{X}$ is a stationary centered Gaussian field.
\item[ii.] Almost surely the paths of $X(t)$ are of class $\mathcal{C}^3$.
\item[iii.] $\textnormal{Var}(X(t))=1$ and $\textnormal{Var}(X'(t))$ is the identity matrix.
\item[iv.] For all $s\neq t \in B$, the distribution of $(X(s),X(t),X'(s),X'(t))$ does not degenerate.
\item[v.] For all $t\in B,\; \gamma \in \mathcal{S}^{n-1}$, the distribution of $(X(t),X'(t),X''(t)\gamma)$ does not degenerate.
\end{itemize}

We use the following  additional notation and hypotheses.
 \begin{itemize}
 \item $\mathcal{S}^{n-1}$ is the unit sphere in $\R^n$.
 \item $S$ is a compact subset of $B$ at a  positive distance from the boundary $\partial B$ and satisfies some regularity properties (see Definition \ref{d:2}).
 \item $B(t,r)$ is the ball of radius $r$ centered at $t$.
 \item $M_Z$ is the maximum of $X(t)$ on the set $Z \subset \R^n$.
 \item $S^{+\epsilon}$ is the tube around $S$ defined as:
 $$S^{+\epsilon}=\left\{t\in \R^n:\; \textnormal{dist}(t,S)\leq \epsilon\right\}.$$
 \end{itemize}
 
 \section{Main results}\label{sect2}

 Firstly, we define,  the class of parameter sets  $S$ that will be considered in dimension 2.
 
 \begin{definition}[\textbf{Two dimensional sets with piecewise-$\mathcal{C}^2$ boundary}]\label{d:2}
We assume that the compact set $S$ consists of a finite number of connected components of the same nature.  We describe in detail the case where  $S$  has only one connected component. $S$ contains two parts:
\begin{itemize}
\item {(i)} The core  $S_c$.  It  is a manifold with piecewise smooth boundary  of class $\mathcal{C}^2$  in the sense of Takemura and Kuriki  \cite{14} :  $S_c$ is  in a neighborhood  of every point  $t$  locally $\mathcal{C}^2$-diffeomorphic  to  a  section of a cone: the support cone defined by \eqref{loco}.  This cone  can be  $\R^2$ for interior points, a half space for  regular points of the boundary, a convex cone for irregular convex points or a cone  with convex complement for irregular concave points.  Note that this last case in excluded in \cite{14}. 
\item{(ii)}A finite set of disjoint self-avoiding  piecewise $\mathcal{C}^2$  "isolated" curves.  
Each curve is "attached" to $S_c$ by a unique point that can be a regular point or a convex irregular point. 
\end{itemize}
As a particular case  of the case above we include also the case where  the core $S_c$  is empty  and in this case  the second part must consist of a single isolated piecewise $\mathcal{C}^2$  curve.

\end{definition}

See Figures  \ref{figu1}  \ref{figu2}  \ref{figu3}  \ref{figrr}  and  \ref{figu10}  for examples of such sets. The boundary of  $S_c$ consist of a finite number of closed continuous piecewise $\mathcal{C}^2$ curves. One of these curves is the exterior boundary and the others are the boundaries of the holes inside $S$. Each of these curves is parameterized by its arc length using the positive orientation.  The "isolated" curves   are also parametrized by arc length with an unimportant orientation. 
 
 \begin{definition} [\textbf{Concave points and  angles}] \label{d:angles}

Irregular points are the points of the boundary of $S$  where the parametrization of the boundary is no longer $\mathcal{C}^2$. They divide the curves above into a finite number of  $\mathcal{C}^2$  \textit{edges}. An edge of  the boundary of $S_c$ will be referred to as \textit{non-isolated edge} and the other as  \textit{isolated edge}. To limit the number of configurations, we assume that an irregular point belongs to  one  of four  following categories:

\begin{itemize}
\item \textit{Convex binary points}: the intersection of two non-isolated edges and the support cone  defined by (\ref{loco}) is convex. See first example in Figure \ref{f:angles}.
\item \textit{Concave binary points}: as above but the support cone  has a complementary convex. Denote $\beta \in [0,\pi)$ as the discontinuity of the angle of the tangent. See second example in Figure \ref{f:angles}.

\item \textit{Angle points}: they are the intersection of two edges belonging to the same isolated curve. Denote $\beta \in [0,\pi)$ by the discontinuity of the angle (the orientation does not matter) as in Figure \ref{figu1}.
\item \textit{Concave ternary points}: the intersection of two non isolated edges $E_1,E_2$ and one isolated one $E_3$.  See third  example in Figure \ref{f:angles}.
In the main result, these points will be considered with multiplicity two. We associate  two concave angles to each of these points:

- $\beta_1$: the discontinuity of the angle of the tangent when we pass from $E_1$ to $E_3$.

- $\beta_2$: the discontinuity of the angle of the tangent when we pass from $E_3$ to $E_2$.

To obtain a rather simple result, we only consider the concave ternary points such that $\beta_1+\beta_2 \leq \pi$, and we exclude more complicated situations  such as point of order four or the existence of handles, for example. 

\end{itemize}

 Finally, the $\beta$'s described above will be referred to as \textbf{concave angles}.
\end{definition}
 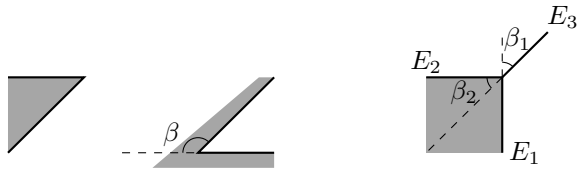
\begin{figure}[h]
\centering
\begin{tikzpicture}
\fill[fill=black!35!white] (-2.5,0)--(-1.5,1)--(-2.5,1)--(-2.5,0);
\draw[thick] (-2.5,0)--(-1.5,1)--(-2.5,1);

\fill[fill=black!35!white] (1,0)--(0,0)--(1,1)--(0.8,1)--(-0.6,-0.2)--(1,-0.2)--(1,0);
\draw[thick] (1,0)--(0,0)--(1,1);
\draw[dashed] (-1,0)--(0,0);
\draw (-0.2,0) arc (180:45:0.2);
\draw (-0.35,0.25) node{$\beta$};

\fill[fill=black!35!white] (4,0)--(4,1)--(3,1)--(3,0)--(4,0);
\draw[thick] (4,0)--(4,1)--(3,1);
\draw[thick] (4,1)--(4.6,1.6);
\draw(4.3,0) node{$E_1$};
\draw(3,1.2) node{$E_2$};
\draw(4.8,1.8) node{$E_3$};
\draw[dashed] (4,1)--(4,1.6);
\draw (4,1.2) arc (90:45:0.2);
\draw (4.2,1.5) node{$\beta_1$};
\draw[dashed] (4,1)--(3,0);
\draw (3.8,1) arc (-180:-135:0.2);
\draw (3.5,0.8) node{$\beta_2$};

\end{tikzpicture}
\caption{Convex, concave binary and concave ternary points, respectively.} \label{f:angles}  
\end{figure}


\textbf{Remark.} It should be observed that the sets with piecewise-$\mathcal{C}^2$ boundary considered here are Whitney stratified manifolds in the sense of \cite[Section 8.1]{1} with some additional restrictions. We refer readers to this book for more details. \medskip

 Our proof will be hereditary proof. We will start  from   a set without concave points and use  the result recalled in the Appendix  to establish an expansion for such a set. We will  then  proceed by union to   extend the result    to the general class  of sets  of Definition 1. To do so, we need  a definition of the property that will be extended by  union. This is the object of the following definition. 

%
\begin{definition}[\textbf{Steiner formula heuristic property}] \label{d:1}A compact subset $S$ of $\R^2$ is said to satisfy the Steiner formula heuristic (SFH) if it satisfies the following conditions:
\begin{itemize}
\item There exist two non-negative constants $L_1(S)$ and $L_0(S)$ such that, as $\epsilon$ tends to $0$,
\begin{equation} \label{ste}
\sigma_2(S^{+\epsilon})=\sigma_2(S)+\epsilon L_1(S)+\pi\epsilon^2 L_0(S)+o(\epsilon^2).
\end{equation}
\item For all processes $X(t)$ satisfying Assumption $A$, 
\begin{equation} \label{ct}
\PP(M_S\geq u)=\sigma_2(S)\frac{u\varphi(u)}{2\pi} + L_1(S)\frac{\varphi(u)}{2\sqrt{2\pi}} + L_0(S)\overline{\Phi}(u)+o\left(u^{-1}\varphi(u)\right),
\end{equation}
as $u\to \infty$.
\end{itemize}
\end{definition}
\textbf{Remarks.} 
\begin{itemize}
\item[1.] There exist some generalizations of the Steiner formula that hold true for every closed set, see \cite{hl}. The present form is more restrictive. 
\item[2.] If $S$ is a convex body, then (\ref{ste})  will take the form  : for all $\epsilon >0$
\begin{equation}\label{stein}
\sigma_2(S^{+\epsilon})=\sigma_2(S)+\epsilon L_1(S)+\pi\epsilon^2 L_0(S).
\end{equation}
$L_1(S)$ is just the perimeter: $\sigma_1(S)$  and $L_0(S)$ is the Euler characteristic of $S$ which is equal to $1$.

If, in addition, the number of irregular points of $S$  (points where the support  cone is not a half space) is finite, then on the basis of  the result of Adler and Taylor, (\ref{ct}) follows. Thus  a convex body with a finite number of irregular points satisfies the SFH property.
\item[3.] If $S$ has a positive reach in the sense that there exists a positive constant $r$ such that for all $t\in S^{+r}$, $t$ has only one projection on $S$, then (\ref{stein}) is true for all $\epsilon< r$ (see \cite{2}, \cite{federer}). Moreover, if, in addition,  $S$ is a set with a  piecewise-$\mathcal{C}^2$ boundary in the sense of Definition \ref{d:2} then it satisfies $\kappa(S)<\infty$ (where $\kappa(S)$ is defined in (\ref{kap})) and  (\ref{ct}) still holds true (see Appendix).
\item[4.] In the most general cases, the constant $L_1(S)$ is the outer Minkowski content of $S$ ($\textnormal{OMC}(S)$), which is defined, when it exists, by:
$$\sigma_2(S^{+\epsilon})=\sigma_2(S)+\epsilon \textnormal{OMC}(S)+o(\epsilon).$$ 
For more details, see \cite{2}. It corresponds to the definition of the perimeter of a set in convex geometry. It can differ from the length of the boundary of $S$, for example in the case of "the square with whiskers" (see Figure \ref{figu2}).

\begin{figure}[h]
\centering
\begin{tikzpicture}
\draw (0,0) rectangle (2,2) ;
\draw (2,1) -- (3,1);
\draw (0,1) -- (-1,1);
\end{tikzpicture}
\caption{The square with whiskers.} \label{figu2}
\end{figure}
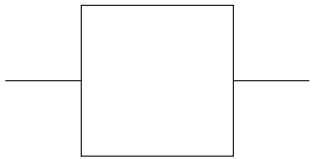

 In this last  case, the length of the boundary is equal to the perimeter of the square plus the length of the whiskers, while $\textnormal{OMC}(S)$ is equal to the perimeter of the square plus {\bf  two times} the length of the whiskers. In addition it should be noticed that $L_0(S)$ is not always equal to the Euler characteristic (see Subsection 3.4). 
 \end{itemize}

 We are  now able to state our main result.
 \begin{theorem}\label{theo} Let $S$ be a compact subset in $\R^2$ with  a piecewise-$\mathcal{C}^2$ boundary and with concave angles $\beta_1,\ldots,\beta_k$ as defined in Definition \ref{d:angles}. Let $\mathcal{X}$ be a random field satisfying Assumption $A$. Let $M_S$ be the maximum of $X(t)$ on $S$.

 Then $S$ satisfies the SFH, more precisely:
$$\sigma_2(S^{+\epsilon})=\sigma_2(S)+\textrm{OMC}(S)\epsilon+\left[ \pi \chi(S)-\sum_{i=1}^k \left( \tan\frac{\beta_i}{2}-\frac{\beta_i}{2}\right)\right]\epsilon^2+o(\epsilon^2),$$ 
  and
\begin{equation}\label{f:th}
\PP(M_S\geq u)=\frac{\sigma_2(S)}{2\pi}u\varphi(u)+\frac{\textnormal{OMC}(S)}{2\sqrt{2\pi}}\varphi(u)+ \left[\chi(S)-\frac{1}{\pi}\sum_{i=1}^k \left(\tan\frac{\beta_i}{2}-\frac{\beta_i}{2} \right)\right]\overline{\Phi}(u) + o\left(u^{-1}\varphi(u)\right),
\end{equation}
where $\chi(S)$ is the Euler characteristic of $S$ that is equal to the number of connected components minus the number of holes.

In addition, the outer Minkowski content $\textnormal{OMC}(S)$ is equal to the length of the non-isolated edges plus twice the length of the isolated edges.
\end{theorem}

 For an illustration of this theorem, see  the examples in Section \ref{sect3}.  Because of Naiman's inequality (see   Naiman \cite{nai} or Johnstone and Sigmund \cite{JS}), the volume of the tube   is always smaller than   it would be in the case  of locally convex sets. As a consequence  the correction term  to the Euler characteristics is always non-positive.

Our starting point in this paper is the following lemma that extend  the  ideas of Aza\"\i s and Wschebor \cite{5} (see Lemma \ref{leaz})   by considering  several sets. 
\begin{lemma}\label{le1} Let $\mathcal{X}$ be a random field satisfying Assumption $A$ and $S_1,\ldots, S_m$ be $m$ subsets of $B$ at  a positive distance from $\partial B$. Assume that there exist two constants $C>0$ and $0\leq d<n$ such that:
\begin{equation}\label{inters}
\sigma_n\left(S_1^{+\epsilon}\cap \ldots \cap S_m^{+\epsilon}\right)= \left(C+o(1)\right)\epsilon^{n-d} \; as \; \epsilon \rightarrow 0,
\end{equation}
where $\sigma_n$ is the Lebesgue  measure in $\R^n$.
Then, as $u\rightarrow +\infty$,
\begin{equation}\label{fracta}
\PP\left(\forall i=1\ldots m : \;M_{S_i}\geq u\right)=u^{d-1}\varphi(u) \left(\frac{C}{2^{d/2}\pi^{n/2}}\Gamma\left(1+(n-d)/2\right)+o(1)\right),
\end{equation}
where $\Gamma$ is the Gamma function.
\end{lemma}
 This lemma is proved  in  Section \ref{s:le1}.
The main idea of the proof of the main theorem  is to use  the inclusion-exclusion principle to compute the probability of the union of events  $\{M_{S_i} >u\}$  through  Lemma \ref{le1}   that gives probability of the intersection of some of them. Let us give  a simple introductory example.  Suppose that  $S=S_1\cup S_2$  with 
$S_1$ and $S_2$ satisfy the SFH as in Definition \ref{d:1}. Suppose, in addition, that  the condition (\ref{inters}) is met, i.e.,
$$\sigma_2(S_1^{+\epsilon}\cap S_2^{+\epsilon})=\left(C+o(1)\right)\epsilon^2.$$
Then,  using Lemma \ref{le1},  we have an  expansion of  $\PP\left(M_{S_1}\geq u,\, M_{S_2}\geq u\right)$  and, by consequence,  an expansion of $\PP(M_S\geq u)$ with  an  error of  $o(u^{-1}\varphi(u))$.

However, in general, we need to  decompose   $S$  into three or even  four sets.  The next lemma is the basis of our method.
 It shows that the Steiner formula heuristic property (SFH)  is heredity in the sense that: if we start from some subsets in $\R^2$ satisfying the SFH, then under some conditions, the union of these subsets also satisfies the SFH. Therefore, to prove the main theorem, we just prove that the considered parameter set can be expressed as the union of the subsets satisfying the SFH.

\begin{lemma}\label{le2} Let $S_1,\;S_2,\; S_3$ and $S_4$ be four compact subsets in $\R^2$ such that:
\begin{itemize}
\item[1.)] For every $i=1,2,3,4$, $S_i$ satisfies the SFH .
\item[2.)] $S_1\cup S_2,\; S_2\cup S_3,\; S_3\cup S_4,\;$ and $S_4\cup S_1$ satisfy the SFH.
\item[3.)] $S_2\cap S_4=\emptyset$ and $S_1\cap S_3 \cap S_4=\emptyset$.
\item[4.)] As $\epsilon$ tends to $0$, there exist two positive constants $C_{13}$ and $C_{123}$ such that
\begin{equation}\label{con}
\sigma_2(S_1^{+\epsilon}\cap S_3^{+\epsilon})= \left(C_{13}+o(1)\right) \epsilon^2 \;\; \textnormal{and} \;\; \sigma_2(S_1^{+\epsilon}\cap S_2^{+\epsilon}\cap S_3^{+\epsilon})= \left( C_{123}+o(1)\right) \epsilon^2.
\end{equation}
\end{itemize}
Then $S=S_1\cup S_2 \cup S_3 \cup S_4$ also satisfies the SFH and:
 \begin{displaymath}
 \begin{array}{rl}
- \indent L_1(S)=&\displaystyle L_1(S_1\cup S_2)+L_1(S_2\cup S_3)+L_1(S_3\cup S_4)+L_1(S_4\cup S_1)- \sum_{i=1}^4L_1(S_i),\\
- \indent L_0(S)=&\displaystyle L_0(S_1\cup S_2)+L_0(S_2\cup S_3)+L_0(S_3\cup S_4)+L_0(S_4\cup S_1)- \sum_{i=1}^4L_0(S_i)+\frac{C_{123}-C_{13}}{\pi}.\\
 \end{array}
 \end{displaymath}
\end{lemma}
 Note that in many cases, Lemma \ref{le2} will be used with $S_4= \emptyset$    and will  consequently take a simpler form. 
\begin{proof} 
In order to prove that $S$ satisfies the SFH, we need to show the correspondence between the tube formula of $S$ as in  (\ref{ste}) and the asymptotic expansion for the tail of the distribution as in (\ref{ct}).

$\bullet$ First, we consider the tube formula of $S$.  We prove the following equality  for the volume of $S^{+\epsilon}$ for a sufficiently small $\epsilon$:
\begin{align} 
A:= & \sigma_2(S^{+\epsilon}) = B:= \sigma_2((S_1\cup S_2)^{+\epsilon})+\sigma_2((S_2\cup S_3)^{+\epsilon})+\sigma_2((S_3\cup S_4)^{+\epsilon})+\sigma_2((S_4\cup S_1)^{+\epsilon}) \notag\\
 & -\sigma_2(S_1^{+\epsilon})-\sigma_2(S_2^{+\epsilon})-\sigma_2(S_3^{+\epsilon})-\sigma_2(S_4^{+\epsilon}) + \sigma_2(S_1^{+\epsilon}\cap S_2^{+\epsilon} \cap S_3^{+\epsilon}) - \sigma_2(S_1^{+\epsilon}\cap S_3^{+\epsilon}).\label{e:inclus}
 \end{align}
  Concerning $A$, we can observe  that $ A =  \sigma_2(S_1^{+\epsilon}   \cup S_2  ^{+\epsilon} \cup S_3 ^{+\epsilon} \cup S_4^{+\epsilon})$ and use the inclusion-exclusion  principle  to obtain  a full expansion.  Doing the same on $B$ we see that the following quantity is missing: 
  $$
  -\{2,4\} +  \{1,2,4\}  +\{2,3,4\} + \{1,3,4\} -\{1,2,3,4\},
$$
where, for example,  $\{2,4\} = \sigma_2(S_2^{+\epsilon}\cap S_4^{+\epsilon})$. Our hypotheses shows that these quantities vanish as soon as $\epsilon$ is small enough. This proves (\ref{e:inclus}) and implies that:

 $$\sigma_2(S^{+\epsilon})=\sigma_2(S)+\epsilon L_1(S)+\pi\epsilon^2 L_0(S)+o(\epsilon^2),$$
 where
 \begin{displaymath}
 \begin{array}{rl}
- \indent L_1(S)=&\displaystyle L_1(S_1\cup S_2)+L_1(S_2\cup S_3)+L_1(S_3\cup S_4)+L_1(S_4\cup S_1)- \sum_{i=1}^4L_1(S_i),\\
- \indent L_0(S)=&\displaystyle L_0(S_1\cup S_2)+L_0(S_2\cup S_3)+L_0(S_3\cup S_4)+L_0(S_4\cup S_1)- \sum_{i=1}^4L_0(S_i)+\frac{C_{123}-C_{13}}{\pi}.\\
 \end{array}
 \end{displaymath}

$\bullet$ For the excursion probability on $S$, using the inclusion-exclusion principle once again,
 \begin{displaymath}
 \begin{array}{rl}
 \PP(M_S\geq u)=& \PP(M_{S_1\cup S_2 \cup S_3 \cup S_4}\geq u)\\
 =& \displaystyle \sum_{i=1}^4 \PP(M_{S_i}\geq u)-\sum_{1\leq i<j \leq 4} \PP(M_{S_i}\geq u,\, M_{S_j}\geq u)\\
 & \displaystyle +\sum_{1\leq i<j<k \leq 4} \PP(M_{S_i}\geq u,\, M_{S_j}\geq u,\, M_{S_k}\geq u)-\PP(M_{S_i}\geq u,\, \forall i=1,2,3,4).
 \end{array}
 \end{displaymath}
 
 On the basis of  Lemma \ref{lebt}, it is easy to see that the events $\{M_{S_2}\geq u,\, M_{S_4}\geq u \}$ and $\{M_{S_1}\geq u,\,M_{S_3}\geq u,\, M_{S_4}\geq u \}$ have negligible probabilities ($o(u^{-1}\varphi(u))$), yielding:
 \begin{align*}
 \PP(M_S\geq u)=& \sum_{i=1}^4 \PP(M_{S_i}\geq u)-\PP(M_{S_1}\geq u,\, M_{S_2}\geq u)-\PP(M_{S_2}\geq u,\, M_{S_3}\geq u)\\
 & \displaystyle -\PP(M_{S_3}\geq u,\, M_{S_4}\geq u)-\PP(M_{S_4}\geq u,\, M_{S_1}\geq u)-\PP(M_{S_1}\geq u,\, M_{S_3}\geq u) \\
 & \displaystyle +\PP(M_{S_1}\geq u,\, M_{S_2}\geq u,\, M_{S_3}\geq u)+o\left(u^{-1}\varphi(u)\right)\\
 = & \displaystyle \PP(M_{S_1}\geq u,\, M_{S_2}\geq u)+\PP(M_{S_2}\geq u,\, M_{S_3}\geq u)+\PP(M_{S_3}\geq u,\, M_{S_4}\geq u)\\
 & +\PP(M_{S_4}\geq u,\, M_{S_1}\geq u) -\displaystyle \sum_{i=1}^4 \PP(M_{S_i}\geq u)-\PP(M_{S_1}\geq u,\, M_{S_3}\geq u) \\
 & \displaystyle +\PP(M_{S_1}\geq u,\, M_{S_2}\geq u,\, M_{S_3}\geq u)+o\left(u^{-1}\varphi(u)\right).
 \end{align*}
 Now, using the SFH property in the first and second conditions and applying Lemma \ref{le1} for two probabilities $\PP(M_{S_1}\geq u,\; M_{S_3}\geq u)$ and $\PP(M_{S_1}\geq u,\; M_{S_2}\geq u,\; M_{S_3}\geq u)$, we can deduce that:
 $$\PP(M_S\geq u)=L_0(S)\overline{\Phi}(u)+L_1(S)\frac{\varphi(u)}{2\sqrt{2\pi}}+\sigma_2(S)\frac{u\varphi(u)}{2\pi}+o\left(u^{-1}\varphi(u)\right),$$
 where the constants $L_0(S)$ and $L_1(S)$ are defined as in the statement.
 
 Since a correspondence  exists between the  two formulas obtained, we have proved the SFH property of $S$.
\end{proof}
%

\textbf{An introductory example to understand the method}

To introduce our method, we consider the case of the simplest non-convex polygon shown in Figure \ref{figu3}. Note that in this case, we have exactly one concave binary point with concave angle $\beta$.

 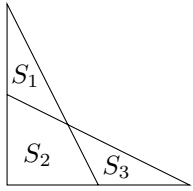
\begin{figure}[h]
\centering
\begin{tikzpicture}[scale=0.8]
\draw (1.5,0)--(0,3) --(0,0)--(3,0)--(0,1.5);
\draw (0.3,1.8) node{$S_1$};
\draw (0.5,0.5) node{$S_2$};
\draw (1.8,0.3) node{$S_3$};
\end{tikzpicture}
\caption{Non-convex polygon with concave binary irregular point.} \label{figu3}
\end{figure}

$S$ is decomposed into three  polygons $S_1,S_2$ and $S_3$ with  a zero measure intersection,  as indicated in Figure \ref{figu3}. These polygons are convex so they satisfy the SFH as well as $S_1\cup S_2$ and $S_2\cup S_3$.

To apply Lemma \ref{le2},  with $S_4 = \emptyset $, it remains to compute the areas of $(S_1^{+\epsilon}\cap S_3^{+\epsilon})$ and $(S_1^{+\epsilon}\cap S_2^{+\epsilon} \cap S_3^{+\epsilon})$. Elementary geometry shows that $(S_1^{+\epsilon}\cap S_3^{+\epsilon})$ consists of two sections of  a disc with angle $(\pi-\beta)$ and two quadrilaterals of area $\epsilon^2\tan(\beta/2)$ each, whereas in $(S_1^{+\epsilon}\cap S_2^{+\epsilon} \cap S_3^{+\epsilon})$, one quadrilateral is replaced by a section of a disc of angle $\beta$ (see Figure \ref{figre}). 

 \begin{figure}[h]
\centering
\begin{tikzpicture}
\fill[fill=black!20!white] (2,0)--(0,0)--(-1.5,1.5)--(-1.5,-0.5)--(2,-0.5)--(2,0);
\draw[thick] (2,0)--(0,0)--(-1.5,1.5);
\draw[dashed] (0.5,0) arc (0:180:0.5);
\draw[dashed] (0.5,0) arc (0:-180:0.5);
\draw[dashed] (-1,0.5)--(2,0.5);
\draw[dashed] (-1,-0.5)--(2,-0.5);
\draw[dashed] (0,-0.5)--(0,0.5);
\draw (-1.5,0)--(0,0);
\draw[dashed, rotate=45] (-0.5,0)--(0.5,0);
\draw[dashed, rotate=45] (0.5,-1)--(0.5,2);
\draw[dashed, rotate=45] (-0.5,-1)--(-0.5,2);
\draw (0,0)--(1,-1);
\draw(-1.4,0.3) node{$S_1$};
\draw (-1,-0.3)node{$S_2$};
\draw(1.5,-0.3) node{$S_3$};
\draw (-0.6,0.2) node{$\beta$};
\end{tikzpicture}
\caption{Intersection of $\epsilon$-neighborhood sets.}\label{figre} 
\end{figure}
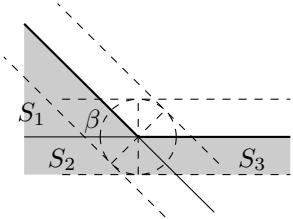

Thus,
$$\sigma_2(S_1^{+\epsilon}\cap S_3^{+\epsilon})=\left[(\pi-\beta)+2\tan \frac{\beta}{2}\right]\epsilon^2,$$
$$\sigma_2(S_1^{+\epsilon}\cap S_2^{+\epsilon} \cap S_3^{+\epsilon})=\left[(\pi-\beta)+\frac{\beta}{2}+\tan \frac{\beta}{2}\right]\epsilon^2.$$

Then using  (\ref{con}), we can define the constants $C_{123}$ and $C_{13}$, and compute that:
$$C_{123}-C_{13}= \frac{\beta}{2}-\tan\frac{\beta}{2}.$$
This quantity measures the non convexity of the concave binary point. Since  the $L_0-$ constants of $S_2,\, S_1\cup S_2$ and $S_2\cup S_3$ are both equal to $1$ in this case, an application of Lemma \ref{le2} shows that the coefficient of $ \overline{\Phi}(u)$ in the expansion of the tail of $M_S$ is now $\displaystyle 1- \frac{\tan(\beta/2)-\beta/2}{\pi}$.  

\subsection*{Proof of the main theorem}
Using the above lemmas, we are able now to prove the main theorem. If the parameter set $S$ consists of several disjoint connected components, then by using  Lemma \ref{lebt} in the appendix,  the tail distribution of the maxima defined on these components can be added with an error of $o\left(u^{-1}\varphi(u)\right)$, and the right-hand side of (\ref{f:th}) is also additive, we can assume in the sequel  that $S$ is connected.

 Our proof is based on induction on the number of concave points of $S$. It should be recalled here that there are three types of concave points: binary, angle and ternary (see Definition \ref{d:angles}).

$\bullet$ Suppose that $S$ has no concave point. Two cases will be considered: depending  on whether $\sr$ is empty or not.

If $\sr$ is empty, then $S$ consists of only one isolated edge. Using the parametrization of the unique edge, we see that $M_S$ is just the maximum of a smooth random process (with parameter of dimension 1). By using the Rice method for the number of up-crossings, Piterbarg \cite{10} or Rychlik \cite{12} showed that $S$ satisfies the SFH. 

If $\sr$ is not empty, then $S$ cannot have isolated edges and  $S$ has a positive reach in the sense of Federer \cite{federer}   because the curvature on the compact edges is bounded. Therefore,
\begin{equation}\label{f:steiner}
\sigma_2(S^{+\epsilon})=\chi(S)\pi \epsilon^2+\textnormal{OMC}(S)\epsilon+\sigma_2(S),
\end{equation}
for  small enough $\epsilon$. On the other hand, on the basis of Theorem 8.12 of Aza\"\i s and Wschebor \cite{6}, it can deduced that the SFH applies (see Appendix for details).

$\bullet$ Suppose  $S$ has at least one concave point. We will decompose  $S$ into the subsets whose  number of concave points is strictly smaller than that of $S$. The induction hypothesis then ensures that these subsets will  satisfy the SFH. We can  therefore use Lemma \ref{le2}  to "glue" them together and  to show that the union $S$ also satisfies the SFH. In fact, our method is based on the  "destruction" of  concave points as in the introductory example. More precisely,  there are four possibilities regarding $P$:
\begin{itemize}
\item[1.] Concave binary point on the exterior boundary of $S$. We decompose $S$ into three  compact subsets $S_1,\, S_2$ and $S_3$ as in Figure \ref{figu3}. 
By decomposition we mean  an essential partition, i.e. that $S= S_1 \cup S_2 \cup S_3$  and that $S_1$, $S_2$ and $S_3$ have disjoint interiors.  The decomposition is as follows: at $P$ we prolong  the two tangents  inward and construct two $\mathcal{C}^2$ paths that avoid all the holes and end at one regular point on the exterior boundary such that these two paths have no intersection other than point $P$. This is always possible because the connected open set $\overset{\circ}{S}$ is path connected. We then define $S_1$, $S_2$ and $S_3$ as in Figure \ref{figrr}. To apply Lemma \ref{le2} with $S_4=\emptyset $, we need to verify all the required conditions. 

To compute $ \sigma_2(S_1^{+\epsilon}\cap S_3^{+\epsilon})$ and 
$\sigma_2(S_1^{+\epsilon}\cap S_2^{+\epsilon} \cap S_3^{+\epsilon})$, we can locally  replace the  edges starting from $P$ by their tangents  with an error of $O( \epsilon^3)$ and thus $o( \epsilon^2)$.
In that case the computation of these areas is exactly the same as in the introductory example.
\begin{figure}[h]
\centering
\begin{tikzpicture}[scale=1.5]
\draw (0,0)--(2,0) arc(0:90:1) arc(0:90:1)--(0,0);
\draw[dashed, rounded corners=4pt] (1,1)--(0.75,1) arc (0:180:0.25)--(0,1);
\draw[dashed] (1,0)--(1,1);
\draw (0.5,1) circle (0.15cm);
\draw (0.5,1.5) node{$S_1$};
\draw (1.5,0.5) node{$S_3$};
\draw (0.5, 0.5) node{$S_2$};
\draw (1.2,1.2) node{$P$};
\end{tikzpicture}
\caption{Decomposition of $S$ at a binary concave point on the exterior boundary.} \label{figrr}
\end{figure}
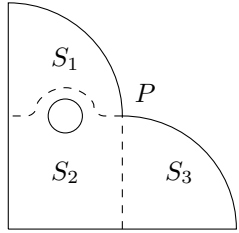
On the other hand, let us consider the number of concave points of $S_1,\, S_2,\, S_3,\, S_1 \cup S_2$ and $S_2\cup S_3$. Due to the way we have constructed these sets, we have destroyed the concavity of $P$ in the sense that with respect to these subsets, $P$ becomes a convex binary point or a regular one.  We can see that an irregular point of these subsets is also an irregular set of $S$ unless it is $P$ or one of the other endpoints on the exterior boundary of the two prolonged paths. However, we have proved that $P$ is no longer a concave point with respect to these subsets. Moreover, since the other endpoint is chosen to be a regular point on the exterior boundary of $S$ and  since the support cones at this point with respect to these subsets (as defined in (\ref{loco})) are included in the one with respect to $S$, these support cones are convex; then this endpoint is therefore  not a concave point. Hence, a concave point of the constructed subsets  is also a concave point of $S$ and $P$ is a concave point only with respect to $S$. These subsets  therefore have a number of concave points equal at most  to the one of $S$ minus 1. They  therefore satisfy the SFH by induction. 

Since all the required conditions are met, on the basis of  Lemma \ref{le2}, $S$ satisfies the SFH with the desired constants.
\item[2.] Concave binary point on the boundary of a hole inside $S$. Drawing two prolonged paths as above, we simply decompose $S$ into two subsets $S_2$ and $S'$. We then  divide $S'$ into three subsets as follows: we also choose two regular points on the boundary of the hole and two corresponding regular points on the exterior boundary of $S$ and construct two smooth curves so that they connect one regular point on the boundary of the hole with the corresponding one on the exterior boundary, and do not intersect themselves or two curves from the irregular point or additional holes. Again, this is possible because $\overset{\circ}{S}$ is path-connected.  $S_1, S_2, S_3, S_4$ are then constructed as in Figure \ref{figu10}. 

 \begin{figure}[h]
\centering
\begin{tikzpicture}
\draw (0,0)--(3,0)--(3,3)--(0,3)--(0,0);
\draw (2,1)--(1,1)--(1,2)--(2,2)--(2,1);
\draw[dashed] (1,0)--(1,1)--(0,1);
\draw[dashed] (1.5,2)--(1.5,3);
\draw[dashed] (2,1.5)--(3,1.5);
\draw (0.5,2.5) node{$S_1$};
\draw (2.5,0.5) node{$S_3$};
\draw (0.5, 0.5) node{$S_2$};
\draw (2.5,2.5) node{$S_4$};
\draw (1.2,1.2) node{$P$};
\end{tikzpicture}
\caption{Decomposition at a concave point on the interior boundary.} \label{figu10}
\end{figure}
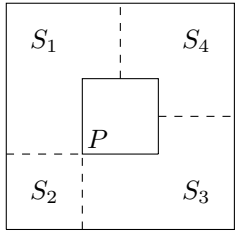

 We will use Lemma \ref{le2} to prove that $S$ satisfies the SFH. Indeed, the computation of the areas $ \sigma_2(S_1^{+\epsilon}\cap S_3^{+\epsilon})$ and 
$\sigma_2(S_1^{+\epsilon}\cap S_2^{+\epsilon} \cap S_3^{+\epsilon})$ and the arguments to show that the subsets in the first and second conditions of  Lemma \ref{le2} satisfy the SFH remain the same as in the case above.  The third condition about the empty intersections is easily verified from the construction. We can  therefore deduce the SFH property of $S$.

\item[3.] Concave ternary point. We define  $S_1$ as the isolated edge containing $P$, $S_2 =\{P\}$ and $S_3$ as the closure of the complement of $S_1$  (see Figure \ref{figg}). 

\begin{figure}[h]
\centering
\begin{tikzpicture}
\draw (0,0) rectangle (2,2) ;
\draw (2,1) -- (3,1);
\draw (1.2,1.2) node{$S_3$};
\draw (2.2,1.2) node{$S_2$};
\draw (3.1,1.2) node{$S_1$};
\end{tikzpicture}
\caption{Decomposition at a concave ternary point.} \label{figg}
\end{figure}
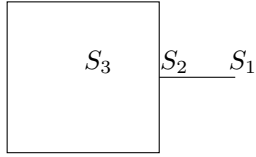

By the same arguments as in the above cases, we can check that all the required conditions in Lemma \ref{le2} are met.  $S$  then satisfies the SFH.
\item[4.] Angle point. We do the same as in the concave ternary point case (see Figure \ref{figgg}).

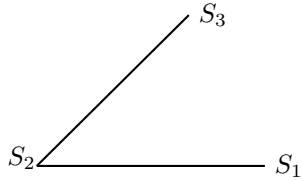
\begin{figure}[h]
\centering
\begin{tikzpicture}
\draw[thick] (0,0) -- (3,0) node[right]{$S_1$};
\draw[thick] (0,0) -- (2,2) node[right]{$S_3$};
\draw (-0.2, 0.1) node{$S_2$};
\end{tikzpicture}
\caption{At an angle point.} \label{figgg}
\end{figure}
\end{itemize}
We have  proved that $S$ satisfies the SFH. In order to establish (\ref{f:th}), we need to compute the constant $L_0(S)$.

Firstly, we have seen that when  $S$ contains no concave points:
$$
L_0(S) =\chi(S)$$
Secondly, Lemma \ref{le2}  shows that when we "glue" $S_1,S_2,S_3$ and $S_4$ together,  each concave points causes a distortion to the  additivity which is equal to:
$$ -\frac{\tan(\beta_i/2)-\beta_i/2}{\pi}.$$
Therefore we  eventually   have:
$$
L_0(S) =\chi(S)  - \sum _{i=1} ^k  \frac{\tan(\beta_i/2)-\beta_i/2}{\pi}$$
 and we are done.

 \section{Examples}\label{sect3}
In this section, we give some examples that are direct applications or direct generalizations of Theorem \ref{theo}. All these results are new and rather unexpected. In most two-dimensional cases, the parameter set  $S$ has the piecewise-$\mathcal{C}^2$ boundary as in Definition \ref{d:2},  satisfying the SFH. Therefore, in order to derive the asymptotic formula for  the tail of the maximum, we just use elementary geometry to compute the area of the tube and consider the corresponding coefficients.
 \subsection{The angle}
 Let $S$ be the angle as in Figure \ref{figu1}. On the basis of  Theorem \ref{theo}, $S$ satisfies the SFH. Using  elementary geometry, we can compute that for  small enough $\epsilon$,
 $$
 \sigma_2(S^{+\epsilon})=2\left(\sigma_1(S_1)+\sigma_1(S_2)\right)\epsilon+\left(\pi+\beta/2-\tan(\beta/2)\right)\epsilon^2,
 $$
 where $\sigma_1(.)$ is simply  the length of the segment. We then have:
 $$\PP(M_{S}\geq u)=\left(1-\frac{\tan(\beta/2)-\beta/2}{\pi}\right)\overline{\Phi}(u)+\frac{\sigma_1(S_1)+\sigma_1(S_2)}{\sqrt{2\pi}}\varphi(u)+o\left(u^{-1}\varphi(u)\right).
 $$
 \subsection{The multi-angle}
This is an extension of the angle case. Let $S$ be a self-avoiding continuous curve that is  the union of $k+1$ curves  with concave angles $\{\beta_1,\ldots ,\beta_k\}$. In this case, the induction process in the proof of the main theorem can be seen as the induction on the number of segments, i.e.,  we add one more segment into the union each time. Here again using  elementary geometry, for small enough $\epsilon$:
 $$\sigma_2(S^{+\epsilon})=2\sigma_1(S)\epsilon+\left(\pi+\underset{i=1}{\overset{k}{\sum}}\left(\beta_i/2-\tan(\beta_i/2)\right)\right)\epsilon^2,$$
where $\sigma_1(S)$ is the length of the curve that is equal to the sum of the lengths of the segments. Hence  we immediately  have the asymptotic formula:
 $$\PP(M_{S}\geq u)=\left(1- \frac{\underset{i=1}{\overset{k}{\sum}}\left(\tan(\beta_i/2)-\beta_i/2\right)}{\pi}\right) \overline{\Phi}(u)+\frac{\sigma_1(S)}{\sqrt{2\pi}}\varphi(u)+o\left(u^{-1}\varphi(u)\right).$$
\subsection{The empty square}
Let $S$ be the empty square, i.e. the boundary of a square in $\R^2$. This case is very similar to the multi-angle case, but the curves are no longer  self-avoiding. In this case, the induction process on the number of segments still works. We can  therefore deduce that  $S$ satisfies the SFH. The elementary geometry shows that for small enough $\epsilon$:
 $$
 \sigma_2(S^{+\epsilon})=2\sigma_1(S)\epsilon+(\pi-4)\epsilon^2;
 $$
 then, as a consequence,
$$\PP(M_S\geq u)= \frac{\pi-4}{\pi}\overline{\Phi}(u)+\frac{\sigma_1(S)}{\sqrt{2\pi}}\varphi(u)+o\left(u^{-1}\varphi(u)\right).$$
\subsection{The full square with whiskers}
We consider "the square with whiskers" as in Figure \ref{figu2}. In this case, $S$ has two concave ternary points. From the main theorem, we know that $S$ satisfies the SFH. Therefore, since we can compute the area of the tube as
$$\sigma_2(S^{+\epsilon})=\sigma_2(S)+\textnormal{OMC}(S)\epsilon+ (2\pi-4)\epsilon^2,$$
for small enough $\epsilon$ ; we have the expansion:
$$\PP(M_S\geq u)=\frac{2\pi-4}{\pi}\overline{\Phi}(u)+\frac{\textnormal{OMC}(S)}{2\sqrt{2\pi}}\varphi(u)+\frac{\sigma_2(S)}{2\pi}u\varphi(u)+o\left(u^{-1}\varphi(u)\right).$$

\subsection{An irregular locally convex set}\label{sub35}
In this subsection, we consider a strange and interesting example. We  consider $S$ as the union of two tangent curves as in Figure \ref{figu7}.
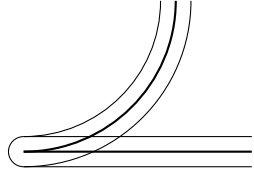
\begin{figure}[h]
\centering
\begin{tikzpicture}
\draw[thick] (0,0) arc (-90:0:2);
\draw[thick] (0,0)--(3,0);
\draw (0,0.2) arc (-90:0:1.8);
\draw (0,-0.2) arc (-90:0:2.2);
\draw (0,0.2)--(3,0.2);
\draw (0,-0.2)--(3,-0.2);
\draw (0,0.2) arc (90:270:0.2);
\end{tikzpicture}
\caption{Two tangent edges.} \label{figu7}
\end{figure}
 
Suppose that $S_1$ is a section of a circle of radius $R$ and $S_2$ is a segment tangent to that circle. For small enough $\epsilon$ , the area of the intersection between two tubes is:

$$\frac{\pi}{2}\epsilon^2+\frac{(R+\epsilon)^2}{2}\arcsin \frac{2\sqrt{R\epsilon}}{R+\epsilon}-(R-\epsilon)\sqrt{R\epsilon}=\frac{\pi}{2}\epsilon^2+\frac{8}{3}\sqrt{R}\epsilon^{3/2}+O(\epsilon^{5/2}).$$
In the above equation, we used the fact that for small enough $x$ ,
$$\arcsin x=x+\frac{1}{2}\frac{x^3}{3}+\frac{1\cdot 3}{2\cdot4}\frac{x^5}{5}+\ldots \; .$$
It is clear that the order of the area of the intersection is not of $2$ as in Condition (\ref{con}), so we cannot apply  Lemma \ref{le2} directly. In this example, the area of the intersection contains two order: $2$ and $3/2$. 

The asymptotic formulas for the tail of the maximum of the random fields defined on $S_1$ and $S_2$ are well understood since they are one-dimensional cases. Then by the inclusion-exclusion principle, 
$$\PP(M_S\geq u)=\PP(M_{S_1}\geq u)+\PP(M_{S_2}\geq u)-\PP\left(M_{S_1}\geq u,\, M_{S_2}\geq u\right).$$
To compute $\PP(M_S\geq u)$, we need to derive an expansion  for $\PP\left(M_{S_1}\geq u,\, M_{S_2}\geq u\right)$.

By carefully  examining  in the proof of Lemma \ref{le1}, we can choose $\alpha$ such that the difference between the upper and the lower bounds of the probability $\PP(M_{S_1}\geq u, \, M_{S_2}\geq u)$ is negligible. Indeed, as in Lemma \ref{le1}, after substituting the area of the intersection of the tubes into the expectation, we obtain the upper bound:
\begin{displaymath}
\begin{array}{rl}
&\displaystyle \frac{1}{2\pi} \int_u^{u+1} \left(x^2+O(1)\right)\varphi(x) \left[\frac{\pi}{2}\frac{2(x-u)}{u-u^{\alpha}}+\frac{8}{3}\sqrt{R}\left(2\frac{x-u}{u-u^{\alpha}}\right)^{3/4}+O\left(\left(2\frac{x-u}{u-u^{\alpha}}\right)^{5/4}\right)\right] dx +o(u^{-1}\varphi(u))\\
=&\displaystyle \frac{1}{2\pi} \int_u^{u+1} x^2\varphi(x) \left[\frac{\pi}{2}\frac{2(x-u)}{u-u^{\alpha}}+\frac{8}{3}\sqrt{R}\left(2\frac{x-u}{u-u^{\alpha}}\right)^{3/4}\right] dx +o(u^{-1}\varphi(u)),
\end{array}
\end{displaymath}

and, similarly, the lower one:
$$\frac{1}{2\pi} \int_u^{u+1} x^2\varphi(x) \left[\frac{\pi}{2}\frac{2(x-u)}{x+u^{\alpha}}+\frac{8}{3}\sqrt{R}\left(2\frac{x-u}{x+u^{\alpha}}\right)^{3/4}\right] dx +o(u^{-1}\varphi(u)).$$

To control the difference between them, we firstly consider the term:
\begin{displaymath}
\begin{array}{rl}
D_1=&\displaystyle \int_u^{u+1} x^2\varphi(x) \left[\left(\frac{x-u}{u-u^{\alpha}}\right)^{3/4}-\left(\frac{x-u}{x+u^{\alpha}}\right)^{3/4}\right] dx\\
=&\displaystyle \int_u^{u+1} x^2\varphi(x) (x-u)^{3/4} \frac{(x+u^{\alpha})^{3/4}-(u-u^{\alpha})^{3/4}}{(x+u^{\alpha})^{3/4}(u-u^{alpha})^{3/4}} dx.
\end{array}
\end{displaymath}

Since 
$$a^{3/4}-b^{3/4}=\frac{a^3-b^3}{\left(a^{3/4}+b^{3/4}\right)\left(a^{3/2}+b^{3/2}\right)}=\frac{(a-b)(a^2+ab+b^2)}{\left(a^{3/4}+b^{3/4}\right)\left(a^{3/2}+b^{3/2}\right)}$$
for $a=x+u^{\alpha}$ and $b=u-u^{\alpha}$,
and we can replace $x,a,b$ by $u$, then:

\begin{displaymath}
\begin{array}{rl}
D_1\leq & \displaystyle (const) \int_u^{u+1} u^2\varphi(x) (x-u)^{3/4} \frac{\left(x-u+2u^{\alpha}\right)u^2}{u^{3+3/4}} dx\\
\leq & \displaystyle (const) u^{\alpha +1/4}\int_u^{u+1} \varphi(x) (x-u)^{3/4} dx.
\end{array}
\end{displaymath}
Here using  the change of variable $x=u+y/u$ once again,
$$D_1 \leq (const) \frac{u^{\alpha +1/4}}{u^{1+3/4}}\varphi(u)\int_0^{u}\exp\left(-y-\frac{y^2}{2u^2}\right)y^{3/4}dy .$$
Therefore, if we choose $\alpha <1/2$ then $D_1=o(u^{-1}\varphi(u))$. For the second term :
$$\int_u^{u+1} x^2\varphi(x) \left[\frac{x-u}{u-u^{\alpha}}-\frac{x-u}{x+u^{\alpha}}\right] dx,$$
we can use the same arguments. Note that this case is simpler.

In conclusion, we have proved that if $\alpha <1/2$, then the difference between the upper and the lower bounds of the probability $\PP(M_{S_1}\geq u, \, M_{S_2}\geq u)$ is negligible. As in Lemma \ref{le1}, we have the following expansion:
$$\PP(M_{S_1}\geq u, \, M_{S_2}\geq u) = \frac{8\sqrt{R}}{2^{1/4}3\pi}\Gamma(7/4)u^{-1/2}\varphi(u)+\frac{\overline{\Phi}(u)}{2} +o(u^{-1}\varphi(u)).$$
Thus, we have:
\begin{proposition} Using the above notation
\begin{equation}\label{tangent}
\PP(M_ {S_1\cup S_2}\geq u)=\frac{3\overline{\Phi}(u)}{2}-\frac{8\sqrt{R}}{2^{1/4}3\pi}\Gamma(7/4)u^{-1/2}\varphi(u)+\frac{\sigma_1(S_1)+\sigma_1(S_2)}{\sqrt{2\pi}}\varphi(u)+o\left(u^{-1}\varphi(u)\right).
\end{equation}
\end{proposition} 

 This example is an apparent counter-example to the results of Adler and Taylor. More precisely, $S$ is clearly a piecewise smooth locally convex manifold: it is easy to check that at the intersection of the circle and the straight line, the support cone is limited to one direction and is  thus convex. Thus  if the random field $X$ is sufficiently smooth, it seems that Theorem 14.3.3 of \cite{1} implies the validity of the Euler characteristic heuristic and Theorem 12.4.2 of \cite{1} gives an expansion of the Euler characteristic function  that should apply. This would be clearly in contradiction with the term $u^{-1/2}\varphi(u)$ in (\ref{tangent}).
 
 In fact, there is no contradiction: Theorem 14.3.3   also demands the manifold to be regular in the sense of Definition 9.22 of \cite{1} and the present set is not a cone space in the sense of Definition 8.3.1 of \cite{1}. This shows that the local convexity itself is not sufficient.
 
 It is surprising to see that in (\ref{tangent}), the asymptotic formula contains three  terms corresponding to the powers: $-1$ (in $\overline{\Phi}(u)$), $-1/2$ and $0$. This is the first time we can see such a combination; in all the well-known cases before, we only saw a combination of integer powers. We emphasize that this strange combination comes from the tube formula of the parameter set.
 
 \subsection{Some examples in dimension 3}\label{sub36}
 Lemmas \ref{le1} and  \ref{le2}  can be applied  in  higher dimensions. However, in dimension 3, for example, they do not make it possible to obtain a full Taylor expansion that would contain, in general, four terms. 
 
 In fact, the coefficient of $\overline{\Phi}(u)$ cannot be determined for non-locally convex sets. We give some examples below.
\begin{itemize}
\item $S$ is a dihedral that is the union of two non-coplanar rectangles $S_1$ and $S_2$, with a common edge such that the angle of the dihedral is $\alpha$, see Figure \ref{figu11}.

 \begin{figure}[h]
\centering
\begin{tikzpicture}
\draw (0,0)--(2,0)--(3,1)--(1,1)--(1,3)--(0,2)--(0,0);
\draw (0,0)--(1,1);
\draw (0.5,0.5)--(0.8,0.5) arc (0:90:0.3)--(0.5,0.5);
\draw (0.7,0.88) node{$\alpha$};
\draw (1.7,0.5) node{$S_1$};
\draw (0.5, 1.7) node{$S_2$};
\end{tikzpicture}
\caption{Example of  a dihedral.} \label{figu11}
\end{figure}
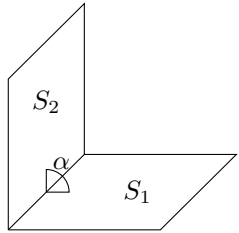

Using the inclusion-exclusion principle, we are  just concerned  with  the probability $\PP\left(M_{S_1}\geq u,\, M_{S_2}\geq u\right)$. Using Lemma \ref{le1} in the case where  $n=3$ and $d=1$, we obtain the expansion of this probability with only one term and an error of $o\left(\varphi(u)\right)$. Then,  
\begin{displaymath}
\begin{array}{rl}
 \PP(M_S\geq u)=& \quad \displaystyle \frac{\sigma_1(\partial S_1)+\sigma_1(\partial S_2)-\sigma_1( S_1\cap S_2)(( \pi+\alpha)/2+\cot(\alpha/2))/\pi}{2\sqrt{2\pi}}\varphi(u)\\
 &\displaystyle +\; \frac{\sigma_2(S_1)+\sigma_2(S_2)}{2\pi}u\varphi(u)+o\left(\varphi(u)\right).
 \end{array}
\end{displaymath}
\item $S$ has the $L-$shape, as in Figure \ref{figu12}.

 \begin{figure}[h]
\centering
\begin{tikzpicture}[scale=0.6]
\draw (0,0)--(2,0)--(2,1)--(1,1)--(1,3)--(0,3)--(0,0);
\draw (2,0)--(2.6,0.4)--(2.6,1.4)--(1.6,1.4)--(1.6,3.4)--(0.6,3.4)--(0,3);
\draw (1,1)--(1.6,1.4);
\draw (1,3)--(1.6,3.4);
\draw (2,1)--(2.6,1.4);
\draw[dashed] (0,0)--(0.6,0.4)--(0.6,3.4);
\draw[dashed] (0.6,0.4)--(2.6,0.4);
\draw[dotted] (1,0)--(1.6,0.4)--(1.6,1.4)-- (0.6,1.4)--(0,1)--(1,1)--(1,0);
\end{tikzpicture}
\caption{L-shape.} \label{figu12}
\end{figure}

Then, by decomposing $S$ into three hyper-rectangles $S_1,\; S_2$ and $S_3$ that are indicated by the dotted lines  with $S_3$  between the two others, we can apply Lemma \ref{le2} with a slight modification that since, in this case, $n=3$ and $d=1$, then the asymptotic formulas for $\PP\left(M_{S_1}\geq u,\, M_{S_2}\geq u\right)$ and $\PP\left(M_{S_1}\geq u,\, M_{S_2}\geq u,\, M_{S_3}\geq u\right)$ are of order $\varphi(u)$ and the error  is $o\left(\varphi(u)\right)$ (from Lemma \ref{le1}). We then  have an expansion with three terms as follows:

\begin{equation} \label{dim3}
 \PP(M_S\geq u)=\frac{\varphi(u)L_1(S)}{2\sqrt{2\pi}}+\frac{L_2(S)u\varphi(u)}{2\pi}+\frac{L_3(S)(u^2-1)\varphi(u)}{(2\pi)^{3/2}}+o\left(\varphi(u)\right),
\end{equation}
 where the coefficients $\{L_i(S),\; i=1,\ldots, 3\}$ are given by the Steiner formula and will be defined  at the end of this section. 

 \item In a more complicated case, i.e., non-convex trihedral (see Figure \ref{figu13}).

 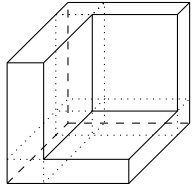
\begin{figure}[h]
\centering
\begin{tikzpicture}[scale=0.8]
\draw (0,0)--(2,0)--(3,1)--(3,3)--(1,3)--(0,2)--(0,0);
\draw (2,0)--(2,0.4)--(0.6,0.4)--(0.6,2)--(1.4,2.8)--(2.8,2.8)--(2.8,1.2)--(2,0.4);
\draw (2.8,1.2)--(1.4,1.2)--(1.4,2.8);
\draw (0.6,0.4)--(1.4,1.2);
\draw (0,2)--(0.6,2);
\draw (2.8,2.8)--(3,3);
\draw[dashed] (0,0)--(1,1)--(1,3);
\draw[dashed] (1,1)--(3,1);
\draw[dotted] (0.8,0.8)--(0.8,2.8)--(1.4,2.8)--(1.6,3)--(1.6,1);
\draw[dotted] (0,0.4)--(0.6,0.4)--(0.6,0);
\draw[dotted] (0.8,2.8)--(1.4,2.8)--(1.6,3);
\draw[dotted] (1.4,1.2)--(0.8,1.2)--(1,1.4)--(1.6,1.4)--(1.4,1.2)--(1.4,0.8);
\draw[dotted] (0.8,0.8)--(2.8,0.8)--(2.8,1.2)--(3,1.4)--(1.6,1.4);
\draw[dotted] (0.6,0)--(1.6,1);
\draw[dotted] (0,0.4)--(0.8,1.2);
\end{tikzpicture}
\caption{Example of  a non-convex trihedral.} \label{figu13}
\end{figure}

In this case, we have three concave edges in the sense that the angles inside the trihedral at these edges are strictly greater than $\pi$. We will destroy this concavity by extending the planes (faces) containing these edges so that they decompose $S$ into smaller convex subsets $\{S_i\}$ with disjoint interiors (see the dotted lines in the figure). Observe that the intersection between two subsets is of one of four types: empty set, a single point, an edge or a face. If it is a face, then the union of these two subsets is also convex. An intersection between three or more subsets is one of three types: empty set, a single point or an edge. Using the inclusion-exclusion principle, we need to find the expansion of the probability of the intersection of the events $M_{S_{i_k}}\geq u$ for some $k$. Concerning the intersection of the $\{S_{i_k}\}$, we have the following cases: 
\begin{itemize}
\item[1.] Empty set. On the basis  of Lemma \ref{lebt}, the probability of the intersection of the events $M_{S_{i_k}}\geq u$ is $o(u^{-1}\varphi(u))$.
\item[2.] A single point. By applying Lemma \ref{le1} in the case $d=0$, the probability considered  is also $o(u^{-1}\varphi(u))$.
\item[3.] An edge. By applying Lemma \ref{le1} in the case $n=3$ and $d=1$, the expansion for the probability  considered is of the order $\varphi(u)$ with the error $o(\varphi(u))$.
\item[4.] A face. This case just happens when we consider the intersection between two subsets. Since both  of these subsets and their union are convex, the expansion for the tail distribution of the maxima defined on them is well-known. We can  therefore compute the expansion for the probability  considered by the inclusion-exclusion principle.
\end{itemize}

We  therefore obtain  an asymptotic formula for $\PP(M_S\geq u)$, as  in (\ref{dim3}).
\end{itemize}
In general, by the same arguments and using  induction, when $S$ is a polytope, 
$$ \PP(M_S\geq u)=\frac{\varphi(u)L_1(S)}{2\sqrt{2\pi}}+\frac{L_2(S)u\varphi(u)}{2\pi}+\frac{L_3(S)(u^2-1)\varphi(u)}{(2\pi)^{3/2}}+o\left(\varphi(u)\right),$$
where
 \begin{itemize}
 \item[-] $L_3(S)$ is the volume of $S$.
 \item[-] $L_2(S)$ is one half of the surface area.
 \item[-] To compute  $L_1(S)$, we consider two types of edge: convex and concave.
 An  internal dihedral angle is associated  to each edge  $i$. If this angle is less than   or equal to $\pi$, the edge is considered to be  convex  and the angle is denoted by $\alpha_i$. Let $h$ be the number of such edges.    If the angle is larger than $\pi$  the edge is considered to be "concave" and the angle  is denoted by $\beta_i $. Let $k$   be the number of such angles, then:
 
 $$L_1(S)=\sum_{i=1}^h\frac{(\pi-\alpha_i)}{2\pi}l_{i}+\sum_{i=1}^k\frac{\cot (\beta_i/2)}{\pi}l_{i},$$
 where $l_i$ is the length of edge $i$. 
 \end{itemize}
\section*{Conclusion}

 The relation between the  expansion tail of the maximum and the Steiner formula was  first established by Sun \cite{13}  and Takemura and Kuriki  \cite{14} for the isonormal Gaussian process defined on the unit sphere. the basis of the proof was the well-known  relation between the standard Gaussian distribution on  $\R^n$   and the uniform distribution on the sphere. 
 In the  rather different cases considered here, the Steiner formula for the tube  still governs the expansion of the tail of the maximum as if the excursion set was precisely a unique ball with  a random radius. We have not found any counter-example to that principle and  we  therefore conjecture that the result is true for a much wider class of sets than those considered in this paper. 
 
 \section{Appendix}
 \subsection{Proof of Lemma  \ref{le1} }  \label{s:le1}
 For the proof, we need some auxiliary lemmas. Firstly, we recall a well-known result on Gaussian processes  \cite{6}. 

\begin{lemma}[Borel-Sudakov-Tsirelson inequality] Let $\mathcal{X}$ be a centered Gaussian field  almost surely bounded on a parameter set  $Z$ . Then $\E(M_Z)<\infty$, and, for all $u>0$,
$$\PP\left(M_Z-\E(M_Z) \geq u\right) \leq \exp(-u^2/(2\sigma^2_Z)),$$
where $\displaystyle \sigma^2_Z=\sup_{t\in Z} \E(X^2(t))$.
\end{lemma}

An  easy consequence of the  BST inequality is that, for each $\epsilon >0$, there exists a constant $C_{\epsilon}>0$ such that for all $u>0$:
\begin{equation}\label{fer} 
\PP(M_Z \geq u)\leq C_{\epsilon} \exp \left(\frac{-u^2}{2(\sigma^2_Z+\epsilon)}\right).
\end{equation}

We will use the above observation to prove the following lemma.
\begin{lemma}\label{lebt} Let $\mathcal{X}$ be a random field satisfying Assumption $A$. Let $Z_1,\ldots ,Z_k$ be some compact subsets of $B$ such that:  
$$Z_1\cap\ldots \cap Z_k=\emptyset.$$
Then there exist two constants $\theta >1$ and $C$ such that for all $u>0$,
$$\PP\left(M_{Z_1} \geq u,\, \ldots ,\, M_{Z_k}\geq u\right)\leq C.\exp(-\theta u^2/2).$$
\end{lemma}
\begin{proof}
 On the set   $Z:=Z_1\times \ldots \times Z_k$, we  consider  the Gaussian field $Y$ defined by: 
$$Y(t_1,\ldots,t_k)=X(t_1)+\ldots +X(t_k).$$
Then
$$\PP\left(M_{Z_1} \geq u,\, \ldots ,\, M_{Z_k}\geq u\right) \leq \PP \left(\underset{t\in Z}{\sup}Y(t) \geq k.u\right).$$
Applying (\ref{fer}) to the Gaussian field $Y$, we see that  for each $\epsilon >0$, there exists a constant $C_{\epsilon}>0$ such that for all $u>0$:
$$\PP(\underset{t\in Z}{\sup}Y(t) \geq k.u)\leq C_{\epsilon} \exp \left(\frac{-k^2u^2}{2(\sigma^2_Z+\epsilon)}\right),$$
where
$$\sigma^2_Z=\underset{t\in Z}{\sup} \E(Y^2(t))=\underset{t\in Z}{\sup} \E\left[(X(t_1)+\ldots+X(t_k))^2\right].$$
Since $\E(X^2(t_i))=1$, $\E\left(X(t_i)X(t_j)\right)<1$ if $t_i\neq t_j$ and $Z$ is compact, we have $\sigma^2_Z <k^2$. By choosing $\epsilon >0$ such that $k^2>\sigma^2_Z+\epsilon$, the result follows.
\end{proof}
Since  we look  at  a result of the type   (\ref{f:th}),  every event with probability   $o\left(u^{-1}\varphi(u)\right)$ can be neglected and will be called "negligible".  Lemma \ref{lebt}   shows that the  event $ (M_{Z_1} \geq u,\, \ldots ,\, M_{Z_k}\geq u)$ is negligible as $u \to +\infty$.

The following lemma is a recent result of Aza\"\i s and Wschebor \cite{7}.
\begin{lemma}\label{leaz}
Let $X$ be a random field satisfying Assumption $A$ and $\alpha$ be a given real number $0<\alpha <1$. Then the following events are negligible:
\begin{itemize}
\item[] $A_1=\left\{\exists \; \mbox{a local maximum in}\; B \mbox{ with value} \; \geq u+1\right\}$.
\item[] $A_2=\left\{\exists \; \mbox{two or more local maxima in}\; \overset{\circ}{B} \mbox{ with value} \; \geq u\right\}$.
\item[] 
\begin{displaymath}
\begin{array}{rl}
A_3 = & \left\{ \exists \,\mbox{a local maximum}\; t \in \overset{\circ}{B} \right.\\
& \left. \displaystyle \mbox{such that} \; u<X(t)<u+1,\; \min\left\{\gamma^T X''(s)\gamma :\, s\in B(t,u^{-\beta}),\, \gamma=\frac{s-t}{\|s-t\|}\right\}\leq -X(t)-u^{\alpha} \right\},
\end{array}
\end{displaymath} where $\alpha$ and $\beta$ are some  positive constants in $(0,1)$, satisfying  $\beta > (1-\alpha)/2$.
\item[] 
\begin{displaymath}
\begin{array}{rl}
A_4 =& \bigg\{ \exists \,\mbox{a local maximum}\; t \in B \\
& \left. \displaystyle \mbox{such that} \; u<X(t)<u+1,\; \max\left\{\gamma^T X''(s)\gamma :\, s\in B(t,u^{-\beta}),\, \gamma=\frac{s-t}{\|s-t\|}\right\}\geq -X(t)+u^{\alpha} \right\}.
\end{array}
\end{displaymath}
\end{itemize} 
\end{lemma}
Let us comment  on  Lemma \ref{leaz}.  Consider the event  $ \{M>u\} \cap  A_1^c \cap \cdots \cap A_4^c$ 
 that differs  from  the event of interest $ \{M>u\}$ by a negligible probability.
 
 Because  we are in $A_2^c $, there exists  at most in $B$ one  local maximum    with  a value larger than $u$.  This implies that the excursion set:
$$K_u :=\{s\in B: \, X(s) \geq u\}$$
 consists of  one connected component. Moreover, because we are in $A_1^c  \cap  A_4^c  $,  and thanks to a Taylor expansion: 
 $$X(s)=X(t)+\frac{1}{2}\|s-t\|^2\gamma^T X''(\eta)\gamma,$$
  this   component  is included  in:
 \begin{equation} \label{overline}
 B(t,\overline{r}) \ \ ; \mbox{  with } \ \ \overline{r}=\sqrt{2\frac{X(t)-u}{u-u^{\alpha}}}
 \end{equation}
 and where $t$ is the location of the local maximum. 
 
  This implies  that, for $u$ large enough, $t$ lies in $  \br$. Using the fact that we are in $A_3^c$, we obtain, in the same manner,
  $$ B(t,r)  \subset K_u  
  $$
   with 
  $$  
   \underline{r}=\sqrt{2\frac{X(t)-u}{X(t)+u^{\alpha}}}$$
   
    Eventually, we obtain:
$$
 \PP(\exists \,\mbox{a local maxima}\; t \in \overset{\circ}{B}: \, t\in S^{+\underline{r}})+o(u^{-1}\varphi(u))
 \leq \PP(M_S\geq u) 
 \leq \PP(\exists \,\mbox{a local maxima}\; t \in \overset{\circ}{B}: \, t\in S^{+\overline{r}})
 +o(u^{-1}\varphi(u)).
 $$
On the basis of  this observation, Aza\"\i s and Wschebor derived an asymptotic formula for the excursion distribution $\PP(M_S\geq u)$. For more details see \cite{7}.

\subsubsection*{Proof of Lemma \ref{le1}}
\begin{proof}
Using Lemma \ref{leaz} we have the upper-bound 
\begin{displaymath}
\begin{array}{rl}
&\displaystyle \PP\left(\forall i=1\ldots m : \;M_{S_i}\geq u\right)\\
\leq & o\left(u^{-1}\varphi(u)\right)+\PP\left(\exists t\in \overset{\circ}{B}: \; X(.)\textnormal{ \;has\; a\; local\; maximum\; at}\; t, \; X(t)>u,\; t\in \underset{i=1}{\overset{m}{\cap}} S_i^{+\overline{r}} \right)\\
 \leq & o\left(u^{-1}\varphi(u)\right)+\E\left(\textnormal{card}\left\{ t\in \overset{\circ}{B}: \; X(.)\textnormal{ \;has\; a\; local\; maximum\; at}\; t, \; X(t)>u,\; t\in \underset{i=1}{\overset{m}{\cap}} S_i^{+\overline{r}} \right\}\right).
\end{array}
\end{displaymath}

Applying  the Rice formula (see \cite[Chapter 6]{6}),
\begin{align*}
E:=& \E\left(\textnormal{card}\left\{ t\in \overset{\circ}{B}: \; X(.)\textnormal{ \;has\; a\; local\; maxima\; at}\; t, \; X(t)>u,\; t\in \underset{i=1}{\overset{m}{\cap}} S_i^{+\overline{r}} \right\}\right)\\
=&  \int_u^{+\infty}dx
\int_{\overset{\circ}{B}}
\E\left(|\det(X''(t))|\mathbb{I}_{\{X''(t)\preceq 0\}}\mathbb{I}_{\{t\in \underset{i=1}{\overset{m}{\cap}} S_i^{+\overline{r}}\}}\mid X(t)=x,\; X'(t)=0\right)p_{X(t),X'(t)}(x,0)\sigma_n(dt)
\\
=& \frac{1}{(2\pi)^{n/2}}  \int_u^{+\infty}
\sigma_n(\underset{i=1}{\overset{m}{\cap}} S_i^{+\overline{r}^*}) \E\left(|\det(X''(0))|\mathbb{I}_{\{X''(0)\preceq 0\}}\mid X(0)=x,X'(0)=0\right)\varphi(x)\, dx,
\end{align*}
where $X''(0)\preceq 0$ means that the  matrix $X''(0)$ is semi definite negative, $p_{X(t),X'(t)}(x,0)$ is the value of the joint density function of the random vector $(X(t),X'(t))$ at the point $(x,0)$, and $\overline{r}^*$ is the value of $\overline{r}$  given by (\ref{overline}) when $X(t)=x$. We use the stationary property of the field  here and the fact that $X(t)$ and $X'(t)$ are two independent Gaussian vectors.

Using  the following result (see Aza\"\i s and Delmas \cite{3}):
$$\E\left(|\det(X''(0))|\mathbb{I}_{\{X''(0)\preceq 0\}}\mid X(0)=x,X'(0)=0\right)=x^n+ O\left(x^{n-2}\right)\; \textnormal{as} \; x\rightarrow \infty ,$$
and hypothesis (\ref{inters}), we have, since we are in $A_1^c$:
\begin{displaymath}
\begin{array}{rl}
E= & \displaystyle \frac{1}{(2\pi)^{n/2}}\int_u^{u+1}x^n\varphi(x)C\left[2\frac{x-u}{u-u^{\alpha}}\right]^{(n-d)/2}dx+ o\left(u^{d-1}\varphi(u)\right)\\
= & \displaystyle \frac{Cu^{(n+d)/2}}{2^{d/2}\pi^{n/2}}\int_u^{u+1}\varphi(x)(x-u)^{(n-d)/2}dx +o\left(u^{d-1}\varphi(u)\right). \\
\end{array}
\end{displaymath}
By the change of variable $x=u+y/u$,
\begin{displaymath}
\begin{array}{rl}
E= & \displaystyle \frac{C}{2^{d/2}\pi^{n/2}}u^{d-1}\varphi(u)\int_0^{u}\exp\left(-y-\frac{y^2}{2u^2}\right)y^{(n-d)/2}dy+ o\left(u^{d-1}\varphi(u)\right) \\
= & \displaystyle u^{d-1}\varphi(u) \left( \frac{C}{2^{d/2}\pi^{n/2}}\Gamma(1+(n-d)/2)+o(1)\right).
\end{array}
\end{displaymath}
We then obtain the upper bound as above.

For the lower bound, we see that
\begin{displaymath}
\begin{array}{rl}
&\displaystyle \PP\left(\forall i=1\ldots m : \; M_{S_i}\geq u\right)\\
\geq & o\left(u^{-1}\varphi(u)\right)+\PP\left(\exists t\in \overset{\circ}{B}: \; X(.)\textnormal{ \;has\; a\; local\; maximum\; at}\; t, \;  X(t)>u,\; t\in \underset{i=1}{\overset{m}{\cap}} S_i^{+\underline{r}} \right).\\
\end{array}
\end{displaymath}
Set
$$\mathcal{M}^{\underline{r}}=\textnormal{card}\left\{ t\in \overset{\circ}{B}: \; X(.)\textnormal{ \;has\; a\; local\; maximum\; at}\; t, \; X(t)>u,\; t\in \underset{i=1}{\overset{m}{\cap}} S_i^{+\underline{r}} \right\}.$$
It is proven in \cite{11} or \cite{7} that
$$0\leq \E(\mathcal{M}^{\underline{r}})-\PP(\mathcal{M}^{\underline{r}}\geq 1)\leq \E(\mathcal{M}^{\underline{r}}(\mathcal{M}^{\underline{r}}-1))/2\leq \E(\mathcal{M}_u(\mathcal{M}_u-1))/2=o\left(u^{-1}\varphi(u)\right),$$
where 
$$\mathcal{M}_{u}=\textnormal{card}\left\{ t\in \overset{\circ}{B}: \; X(.)\textnormal{ \;has\; a\; local\; maximum\; at}\; t, \; X(t)> u\right\}.$$
Then
$$\PP\left(\min_i\{M_{S_i}\}\geq u\right)\geq o\left(u^{-1}\varphi(u)\right)+\E(\mathcal{M}^{\underline{r}}).$$
Here, using  the Rice formula  again and by the same arguments, we obtain the same equivalent formula for both the upper and lower bounds. The result  then follows.
\end{proof}

 \subsection{ SFH property for sets with positive reach } 
In this section, we prove that a compact connected set in $\R^2$ with piecewise-$ \mathcal{C}^2 $ boundary and without concave irregular points will satisfy the SFH. This is very similar  to the general result of Adler and Taylor, see Theorem 14.3.3 in \cite{1}. However, these authors just clarified and specified this theorem in the convex case and  we think  that there is a need to provide the following proof.

Firstly, the Steiner formula (\ref{f:steiner}) has already been established. We now consider the excursion probability. We recall the following definitions
 \begin{itemize}
\item Let $S_2$ be the interior of $S$; $S_1$ be the union of the $\mathcal{C}^2$ edges and $S_0$ be the union of the convex irregular points.
 \item For $t\in S_j$, $X'_j(t)$ and $X''_j(t)$ are  the first and second derivatives of $X$ along $S_j$ respectively; $X'_{j,N}(t)$ denotes the outward normal derivative. 
 \end{itemize}
 In our case, it is easy to see that:
$$\kappa (S)= \underset{t\in S}{\sup} \underset{s\in S,\; s\neq t}{\sup} \frac{\textnormal{dist}(s-t,C_t)}{\|s-t\|^2}<\infty.$$

In order to apply Theorem 8.12 and Corollary 8.13 of Aza\"\i s and Wschebor \cite{6}, we have to check the conditions (A1) to (A5) (see \cite[p. 185]{6}). The first three are regularity conditions that are included in  Assumption A. Note that since the edges are of dimension 1, a direct proof of  the Rice formula can be performed without assuming that they are of class $\mathcal{C}^3$ as in (A1).
\begin{itemize}
\item The condition (A4) states that the maximum is attained at a single point. It can be deduced from the Bulinskaya lemma (Proposition 6.11 in \cite{6}) since for $s\neq t$, $(X(s),X(t),X'(s),X'(t))$ has a non-degenerate distribution. 
\item The condition (A5)  states that  there is almost surely no point $t\in S$ such that $X'(t)=0$ and $\det\left(X''(t)\right)=0$. It can be deduced from Proposition 6.5 in \cite{6} applied to the process, $X'(t)$, which is $\mathcal{C}^2$.
\end{itemize}

Since all the required conditions are met, we have:

\begin{equation}\label{f:1}
\underset{u\rightarrow +\infty}{\liminf}-2u^{-2 }\log \big[\int_u^{\infty}p^E(x)dx-\PP\{M_S\geq u\}\big] \geq 1+ \underset{t\in S}{\inf}\frac{1}{\sigma_t^2+\kappa_t^2}>1,
\end{equation}
where 
\begin{itemize}
\item $p^E(x)$ is the approximation of the density of the maximum given by the Euler characteristic method. More precisely, 
 \begin{equation}\label{dens}
\begin{array}{rl}
p^E(x)=& \displaystyle \sum_{t\in S_0}\E \left(\mathbb{I}_{X'_0(t)\in \widehat{C}_{t,0}}\mid X(t)=x \right)\varphi(x)\\
&+ \displaystyle \sum_{j=1}^2 (-1)^j\int_{S_j} \E \left(\det\left(X''_j(t)\right)\mathbb{I}_{X'_{j,N}(t)\in \widehat{C}_{t,j}}\mid X(t)=x,\, X'_j(t)=0\right)\frac{\varphi(x)}{(2\pi)^{j/2}}dt,\\
\end{array}
\end{equation}
where $\widehat{C}_{t,j}$ is the dual cone of the support cone $C_t$,
$$\widehat{C}_{t,j}=\{z\in \R^2:\; \langle z,x\rangle \geq 0, \; \forall \; x\in C_t\}.$$ 

\item $ \displaystyle \sigma_t^2=\underset{s\in S\setminus\{t\}}{\sup}\frac{\textnormal{Var}\left(X(s)\mid X(t),X'(t)\right)}{(1-\textnormal{Cov}(X(s),X(t)))^2}.$

\item $  \displaystyle\kappa_t=\underset{s\in S\setminus\{t\}}{\sup}\frac{\textnormal{dist}\left(\frac{\partial}{\partial t} \textnormal{Cov}(X(s),X(t)),C_t\right)}{1-\textnormal{Cov}(X(s),X(t))}.$
\end{itemize}

We compute $p^E(x)$ as follows:
\begin{itemize}
\item When $j=2$, there is no normal space and $ X'_{2,N}(t)$ makes no sense. It is easy to see that (see, for example, Aza\"\i s and Wschebor \cite[p. 244]{6})
$$\int_{S_2} \E \left(\det\left(X''_2(t)\right)\mid X(t)=x,\, X'_2(t)=0\right)dt = \sigma_2(S) (x^2-1).$$

\item When $j=0$, $X'_{0,N}(t)\- = X'(t)$ and:
$$\E \left(\mathbb{I}_{X'(t)\in \widehat{C}_{t,0}}\mid X(t)=x \right)=\frac{\mathcal{A}(\widehat{C}_{t,0})}{2\pi};$$
where $\mathcal{A}(\widehat{C}_{t,0})$ is the angle of the cone that is equal to the discontinuity of the angle of the tangent at the irregular point $t$.

\item When $j=1$, we consider a point $t$ on an edge $L$ of the exterior boundary. Note that in this case, the support cone $C_t$ is just a half-plane, so the event $\{X'_{1,N}(t)\in \widehat{C}_{t,1}\}$ can be viewed as $\{X'_{1,N}(t) \geq 0\}$.

 At the point $t$, the second derivative along the curve can be expressed as:
$$X''_1(t)=X''_T(t)+C(t)X'_{1,N}(t),$$
where $X''_T$ is  the second derivative  in the tangent direction and $C(t)$ is the signed curvature at $t$.

 It is easy to check that the covariance function of the vector $(X''_T,X'_{1,N},X, X_1')$ is:
 $$\left( \begin{array}{cccc}
 \textnormal{Var}(X''_T) & 0&-1&0\\
 0&1&0&0\\
 -1&0&1&0\\
 0&0&0&1\\
 \end{array} \right).$$
 Therefore, for such  an edge $L$,
$$
 \E \left(X''_1(t)\mathbb{I}_{X'_{1,N}(t)\in \widehat{C}_{t,1}}\mid X(t)=x,\, X'_1(t)=0\right)
 = \E\left(\left(-x+C(t)X'_{1,N}(t)\right)\mathbb{I}_{X'_{1,N}(t)\in \widehat{C}_{t,1}} \right)
 = \frac{-x}{2}+\frac{C(t)}{\sqrt{2\pi}}
$$
and
$$-\int_{L}\E \left(X''_1(t)\mathbb{I}_{X'_{1,N}(t)\in \widehat{C}_{t,1}}\mid X(t)=x,\, X_1'(t)=0\right)\frac{\varphi(x)}{\sqrt{2\pi}}dt=\frac{\sigma_1(L)x}{2\sqrt{2\pi}}\varphi(x)-\frac{\varphi(x)}{2\pi}\int_{L}C(t)dt.$$
The quantity $-\int_{L}C(t)dt$ can be viewed as the variation of the angle of the tangent from the beginning to the end of this edge.

Since we complete a whole turn in the positive orientation:
$$\sum_{\textnormal{ irregular points of the ext. boundary}}\mathcal{A}(\widehat{C}_{t}) \quad +\sum_{\textnormal{edges of the ext. boundary}}-\int_{L_i}C(t)dt=2\pi.$$
For a point $t$ on an edge $L_i$ of the interior boundary (holes), the interpretation of the second derivative changes into:
$$X''_1(t)=X''_T(t)-C(t)X'_{1,N}(t).$$
Therefore,
$$-\int_{L_i}\E \left(X''_1(t)\mathbb{I}_{X'_{1,N}(t)\in \widehat{C}_{t,1}}\mid X(t)=x,\, X_1'(t)=0\right)\frac{\varphi(x)}{\sqrt{2\pi}}dt=\frac{\sigma_1(L_i)x}{2\sqrt{2\pi}}\varphi(x)+\frac{\varphi(x)}{2\pi}\int_{L_i}C(t)dt.$$

\end{itemize}

For the boundary of a hole inside $S$,
$$\sum_{\textnormal{irregular points}}\mathcal{A}(\widehat{C}_{t})+\sum_{\textnormal{edges}}\int_{L_i}C(t)dt=-2\pi.$$

In conclusion, substituting into (\ref{dens}), 
$$p^E(x)=\chi(S)\varphi(x)+\frac{\sigma_1(\partial S)}{2\sqrt{2\pi}}x\varphi(x)+\frac{\sigma_2(S)}{2\pi}(x^2-1)\varphi(x),$$
since the Euler characteristic $\chi(S)$ is equal to $1$ (the number of connected components) minus the number of the holes.

 Integrating $p^E(x)$, we obtain the asymptotic expansion:
$$\PP(M_S \geq u)=\chi(S)\overline{\Phi}(u)+\frac{\sigma_1(\partial S)}{2\sqrt{2\pi}}\varphi(u)+\frac{\sigma_2(S)}{2\pi}u\varphi(u) +Rest,$$
 where $Rest$ is super-exponentially smaller in the sense of (\ref{f:1}).
 This  implies a correspondence between the asymptotic expansion and the Steiner formula.
 
 \indent  \textbf{Acknowledgements:} We would like to thank an  anonymous reviewer for very constructive remarks. The second author is funded by Vietnam National Foundation for Science and Technology Development (NAFOSTED) under grant number 101.03-2014.14.

jean-marc.azais@math.univ-toulouse.fr\\
pgviethung@gmail.com 
\end{document}